\documentclass{birkjour}
\usepackage{subfigure}
\usepackage{verbatim}
 \newtheorem{thm}{Theorem}[section]
 \newtheorem{corollary}[thm]{Corollary}
 \newtheorem{lemma}[thm]{Lemma}
 \newtheorem{Proposition}[thm]{Proposition}
 \theoremstyle{definition}
 
 \theoremstyle{remark}
 \newtheorem{remark}[thm]{Remark}
 \newtheorem{example}{Example}
 \numberwithin{equation}{section}

 \newcommand{\R}{\mathbb{R}}
  \newcommand{\N}{\mathbb{N}}

    \renewcommand{\H}{\mathcal{H}}
    
\renewcommand{\L}{\mathcal{L}}

\begin{document}

%-------------------------------------------------------------------------
% editorial commands: to be inserted by the editorial office

\title[Cycloids in a Normed Plane]
 {Closed Cycloids in a Normed Plane }

%----------Author 1
\author[M.Craizer]{Marcos Craizer}

\address{%
Departamento de Matem\'{a}tica- PUC-Rio\br
Rio de Janeiro, RJ, Brasil}
\email{craizer@puc-rio.br}

\author[R.C.Teixeira]{Ralph Teixeira}

\address{%
Departamento de Matem\'atica Aplicada- UFF\br
Niter\'{o}i, RJ, Brasil}
\email{ralph@mat.uff.br}

\author[V.Balestro]{Vitor Balestro}

\address{%
Instituto de Matem\'{a}tica e Estat\'{i}stica- UFF\br
\&\br
CEFET/RJ- campus Nova Friburgo, Brasil}
\email{vitorbalestro@id.uff.br}

\thanks{The first named author wants to thank CNPq for financial support during the preparation of this manuscript. \newline E-mail of the corresponding author: craizer@puc-rio.br}



\date{April 25, 2016}

\begin{abstract}
Given a normed plane $\mathcal{P}$, we call $\mathcal{P}$-cycloids the planar curves which are
homothetic to their double $\mathcal{P}$-evolutes. It turns out that the radius of curvature and the support function of a $\mathcal{P}$-cycloid satisfy 
a differential equation of Sturm-Liouville type. By studying this equation we can describe all closed hypocycloids and epicycloids with a given number of cusps. 
We can also find an orthonormal basis of ${\mathcal C}^0(S^1)$ with a natural decomposition into symmetric and anti-symmetric functions, 
which are support functions of symmetric and constant width curves, respectively. 
As applications, we prove that the iterations of involutes of a closed curve converge to a constant and a generalization of the 
Sturm-Hurwitz Theorem. We also prove versions of the four vertices theorem for closed curves and six vertices theorem for closed constant width curves.
\end{abstract}

%%% ----------------------------------------------------------------------
\maketitle
%%% ----------------------------------------------------------------------
%\tableofcontents

\section{Introduction}

Euclidean cycloids are planar curves homothetic to their double evolutes. 
When the ratio of homothety is $\lambda>1$, they are called hypocycloids, and when $0<\lambda<1$, they are called epicycloids.   In a normed plane, or Minkowski plane (\cite{Ma-Sw-We, Ma-Sw, Ma-Se, Thompson96}),
there exists a natural way to define the double evolute of a locally convex curve (\cite{Craizer14}). 
Cycloids with respect to a normed plane $\mathcal{P}$ are curves homothetic to their double $\mathcal{P}$-evolutes. Generalizing the Euclidean case,
when the ratio of homothety is $\lambda>1$, we call them $\mathcal{P}$-hypocycloids and when $0<\lambda<1$, we call them $\mathcal{P}$-epicycloids. 

The main tool for studying $\mathcal{P}$-cycloids is the differential equation  
\begin{equation}\label{eq:DiffEq}
\frac{1}{[p,p']} \left(  \frac{ u' }{[q,q']  }  \right)'=-\lambda u(\theta).
\end{equation}
where $p$ is a parameterization of the unit circle of the normed plane ${\mathcal P}$, $q$ the dual parameterization and $[\cdot,\cdot]$ denote the determinant of a pair of vectors in $\R^2$. 
It turns out that a curve is a ${\mathcal P}$-cycloid if and only if its radius of curvature and its support functions satisfy equation \eqref{eq:DiffEq}. This equation is 
of Sturm-Liouville type, and there is a huge amount of theory concerning Sturm-Liouville equations (\cite{Zettl}). For a study of a similar second-order linear differential equation
associated to the geometry of normed planes, see \cite{Petty2}.

We are interested in characterizing the closed cycloids in a general normed plane ${\mathcal P}$. The values of $\lambda$ for which there exists a $2\pi$-periodic solution $u$ of equation \eqref{eq:DiffEq}
are called eigenvalues and the corresponding 
solutions $u$ are called eigenvectors. We shall prove the existence an infinite sequence of eigenvalues $\{\lambda_k^i\}$ and eigenvectors $\{h_k^i\}$, $k\geq 1$, $i=1,2$, 
$$
\lambda_0=0<\lambda_1^1=\lambda_1^2=1<\lambda_2^1\leq\lambda_2^2<\lambda_3^1\leq\lambda_3^2<....
$$
The eigenvectors determine closed hypocycloids, except the ones associated with the eigenvalue $1$, which determines
a $2$-dimensional space of non-closed $\mathcal{P}$-cycloids. Moreover, each $\mathcal{P}$-hypocycloid is regular except for exactly $2k$ ordinary cusps. 

Although the results of the above paragraph can be obtained from the general results of Sturm-Liouville theory (see \cite{CodLev}, ch.8), we give geometric proofs of them. From these proofs, 
we can also understand the epicycloids and hypocycloids that close after $N$ turns. In particular, we prove that, for $l=1,...,N-1$ there is exactly one $2$-dimensional space of $N$-periodic epicycloids
with $2l$ ordinary cusps.

The eigenvalues $\lambda_k^i$, $k\in\N$, $i=1,2$ are characteristic values of the unit ball of the normed plane $\mathcal{P}$. Thus they are obtained from any 
convex symmetric planar set. We show that rotational symmetry of this convex set implies that $\lambda_k^1=\lambda_k^2$, for $k$ odd. It would be interesting to understand
under what conditions duplicity of eigenvalues implies symmetry. In particular, a question that we could not answer is whether or not the duplicity of all eigenvalues implies 
that the normed plane is Euclidean. 

In this paper, we shall consider only normed planes whose unit circle is a quadratically convex symmetric smooth curve. Except for the cycloids corresponding to the double eigenvalue $\lambda=1$, 
we cannot solve explicitly equation \eqref{eq:DiffEq}, and so it is hard to give explicit examples of non-Euclidean cycloids. In \cite{CTB} we consider normed planes whose unit circle is 
a convex symmetric polygon, where we can construct explicit examples of cycloids. These discrete cycloids can also be seen as approximations of the smooth cycloids obtained from
the solutions of equation \eqref{eq:DiffEq}. In fact, the illustrative cycloid drawings of figures \ref{fig:ClosedCyc} and \ref{fig:UnboundedCyc} were obtained by this method.

The eigenvectors $h_k^i$ form a basis of ${\mathcal C}^0(S^1)$ orthogonal with respect to an adapted inner product. 
The space ${\mathcal C}^0(S^1)$ is naturally decomposed in orthogonal subspaces: The constant support functions form 
a $1$-dimensional subspace corresponding to multiples of the unit ball $\mathcal{P}$, while its orthogonal complement
corresponds to curves with zero dual length. 
The subspace $E$ of symmetric support functions is orthogonal to the subspace $W_0$ of anti-symmetric functions.
$E$ corresponds to symmetric curves and the subspace $W_0+K$ corresponds to constant width curves. We recall that sum of
support functions correspond to Minkowski sum of the corresponding curves. 

As applications of the above decomposition, we prove that the iteration of involutes of a zero dual length curve converges to a constant curve
%This result was proved geometrically in \cite{Craizer14} for constant width curves. 
and a generalization of the Sturm-Hurwitz Theorem, 
which says that functions whose first $k$ harmonics are zero must have at least $2(k+1)$ zeros. We prove also a four vertices theorem 
for closed curves and a six vertices theorem for constant width curves.

The paper is organized as follows: In section 2 we show that the support function and the curvature radius of a ${\mathcal P}$-cycloid satisfy equation \eqref{eq:DiffEq}. In section 3
we discuss the basic properties of equation \eqref{eq:DiffEq}, in particular showing that it admits a double eigenvalue $1$. In section 4 we prove the main results of the paper 
concerning the existence of infinite sequence $\gamma_k^i$, $k\in\N$, $i=1,2$, of  closed cycloids, each curve $\gamma_k^i$ with exactly $2k$ cusps. In section 5 we show that
the corresponding support functions $h_k^i$, $k\in\N$, $i=1,2$, form a basis of ${\mathcal C}^0(S^1)$, prove the convergence of the iteration of involutes of a curve and the generalization of Sturm-Hurwitz Theorem, and show versions of the four and six vertices theorem. 

\section{Evolutes, curvature and support functions }

\subsection{A class of Legendrian immersions}

Let $I\subset\R$ be an interval and denote by $\mathrm{P}\R^1$ the projective line. A smooth curve $(\gamma,\nu):I\to\R^2\times \mathrm{P}\R^1$ is {\it Legendrian} if $[\nu(t),\gamma'(t)]=0$. The projection $\gamma$ of a Legendrian immersion 
in $\R^2$ is called a {\it front} (\cite{Takahashi1}).

\begin{lemma}\label{lemma:Equivalences1}
Let $\gamma$ be a front. Then the following statements are equivalent:
\begin{enumerate}
\item\label{item:vlinha} $\nu'(t)\neq 0$.
\item\label{item:env} $\gamma$ is the envelope of its tangent lines.
\item\label{item:tv} $\gamma$ can be parameterized by the angle $\theta$ that its tangents make with a fixed direction.
\end{enumerate}
\end{lemma}

\begin{proof}
Consider the tangent line $[X-\gamma(t),\nu(t)]=0$. Differentiating we obtain $[X-\gamma(t),\nu'(t)]=0$. These equations have a unique solution $X=\gamma(t)$ if and only if
$[\nu(t),\nu'(t)]\neq 0$. This is equivalent to $\nu'(t)\neq 0$, which proves the equivalence between \ref{item:vlinha} and \ref{item:env}. Moreover $\nu'(t)\neq 0$ is equivalent to say that we can write $t=t(\nu)$, which is equivalent to the third statement.
\end{proof}

We shall denote by $\H$ the set of fronts satisfying any, and hence all, of the conditions of Lemma \ref{lemma:Equivalences1}. If $\gamma\in\H$ is closed, it is called 
a {\it hedgehog} (\cite{Martinez-Maure}). In this paper we shall only consider curves $\gamma\in\H$. In particular $\gamma$ has no inflection points. 

\begin{example}\label{ex:Cusps}
Consider $\gamma(t)=\left(t^m,t^n\right)$, $1<m<n$, and take $\nu(t)=(m,nt^{n-m})$. Then $\nu'(0)\neq 0$ if and only if $n=m+1$. In particular, if $\gamma''(0)\neq 0$ then $\gamma(t)=(t^2,t^3)$. 
\end{example}

\subsection{Dual unit circle, support functions and curvature radius}

Let $\mathcal{P}$ be a normed plane with a smooth quadratically convex symmetric unit circle. Parameterize this unit circle by $p(\theta)$, $\theta\in[0,2\pi]$, $p(\theta+\pi)=-p(\theta)$,
where $\theta$ is the angle of the tangent to $p$ with a fixed direction. The above hypothesis imply that $[p,p']>0$ and  $[p',p'']>0$. 
It is not difficult to verify that 
\begin{equation}\label{eq:ParameterQ}
q(\theta)=\frac{p'}{[p,p']}(\theta)
\end{equation}
is a parameterization of the dual unit circle, i.e., $[p,q]=1$, $q\parallel p'$ and $p\parallel q'$ (\cite{Tabach97}). We call $q(\theta)$ the dual parameterization and one can verify that 
\begin{equation}\label{eq:ParameterP}
p(\theta)=-\frac{q'}{ [q,q'] } (\theta)
\end{equation}
and $[p,p']\cdot [q,q']^2=[q,q'']$.

Take $\gamma\in\H$ and parameterize it by the same parameter $\theta$ as above. The real function  $h(\theta)=[\gamma(\theta),q(\theta)]$ is called the $\mathcal{P}$-{\it support function}. Differentiating we obtain $h'=[p,\gamma]\cdot [q,q']$ and so we can write 
\begin{equation}\label{eq:Gamma}
\gamma(\theta)=h(\theta)p(\theta)+\frac{h'}{[q,q']}(\theta)q(\theta).
\end{equation}
Differentiating equation \eqref{eq:Gamma} we get
\begin{equation}\label{eq:GammaLinha}
\gamma'(\theta)= \left( h+\frac{1}{[p,p']} \left( \frac{h'}{[q,q']} \right)' \right)(\theta) p'(\theta).
\end{equation}
The curvature radius $r(\theta)$ of $\gamma$ is defined by the condition 
\begin{equation}\label{eq:CurvatureRadius}
\gamma'(\theta)=r(\theta)p'(\theta),
\end{equation} 
and so we have
\begin{equation}\label{eq:CurvatureSupport}
r(\theta)=h(\theta)+\frac{1}{[p,p']} \left( \frac{h'}{[q,q']} \right)'(\theta).
\end{equation}
For different curvature concepts in a Minkowski plane, see \cite{Petty1}. 

\subsection{Curvature radius and support function of the evolute}

The evolute of $\gamma$ is the curve $\delta$ defined by
\begin{equation*}
\delta(\theta)= \gamma(\theta)-r(\theta) p(\theta). 
\end{equation*}
We conclude that the support function $h_{\delta}=[\delta(\theta),p(\theta)]$ of $\delta$ is given by
\begin{equation*}
h_{\delta}(\theta)=[\gamma(\theta),p(\theta)]=-\frac{h'}{[q,q']}(\theta).
\end{equation*}
Then 
\begin{equation*}
\delta'(\theta)= -r'(\theta) p(\theta)=\frac{ r'}{[q,q'] } (\theta)  q'(\theta),
\end{equation*}
which implies that the curvature radius $r_{\delta}(\theta)$ of the evolute $\delta$ is 
\begin{equation}\label{eq:CurvEvolute}
r_{\delta}(\theta)=\frac{ r' }{[q,q'] }(\theta).
\end{equation}
Observe that the transformations of $h$ and $r$ to obtain $h_{\delta}$ and $r_{\delta}$ are, up to a sign, the same.

\subsection{Double evolutes and the main differential equation}\label{sec:DoubleEvolutes}

Denote by $\eta$ the evolute of $\delta$. Then 
\begin{equation*}
\eta(\theta)=\delta(\theta)-r_{\delta}(\theta) q(\theta)=\gamma(\theta)-r(\theta)p(\theta)-r_{\delta}(\theta) q(\theta).
\end{equation*}
The support function $h_{\eta}$ of $\eta$ is thus
\begin{equation*}\label{eq:SupDoubleEvolute}
h_{\eta}(\theta)=[\eta,q](\theta)=h(\theta)-r(\theta)=-\frac{1}{[p,p']} \left( \frac{h'}{[q,q']} \right)'(\theta).
\end{equation*}
Now
\begin{equation*}
\eta'(\theta)=- r_{\delta}'(\theta)q(\theta)=- \frac{r_{\delta}'}{[p,p']}(\theta) p'(\theta).
\end{equation*}
Thus the curvature $r_{\eta}$ of $\eta$ is 
\begin{equation*}\label{eq:CurvDoubleEvolute}
r_{\eta}=- \frac{1}{[p,p']} \left(  \frac{ r' }{[q,q']  }  \right)'. 
\end{equation*}
Observe that the transformations of $h$ and $r$ to obtain $h_{\eta}$ and $r_{\eta}$ are exactly the same.

We are interested in curves $\gamma$ whose second evolute $\eta$ is homothetic to $\gamma$. For such curves, both the
support function and the curvature radius satisfy the differential equation \eqref{eq:DiffEq}. Next lemma says that equation \eqref{eq:DiffEq} 
is invariant under a re-parameterization of the dual unit circle $q$.  

\begin{lemma}
Write $q(t)=q(\theta(t))$ and $p(t)=p(\theta(t))$ for some parameter $t$ such that $\theta'(t)\neq 0$. If $u(\theta)$ is a solution of equation \eqref{eq:DiffEq} in the parameter $\theta$, 
then $u(t)=u(\theta(t))$ is a solution of equation \eqref{eq:DiffEq} in the parameter $t$.
\end{lemma}
\begin{proof}
Since $q'(t)=q'(\theta)\frac{d\theta}{dt},\ u'(t)=u'(\theta)\frac{d\theta}{dt}$, we have that $\frac{u'}{[q,q']}(t)=\frac{u'}{[q,q']}(\theta)$. A similar argument shows that
$\frac{1}{[p,p']}\left( \frac{u'}{[q,q']}\right)' (t)=\frac{1}{[p,p']}\left( \frac{u'}{[q,q']}\right)' (\theta)$, thus proving the lemma.
\end{proof}

Assume that the support function $h$ of a curve $\gamma$ satisfies equation \eqref{eq:DiffEq}. Then equation \eqref{eq:CurvatureSupport} implies that $r=h(1-\lambda)$
also satisfies equation \eqref{eq:DiffEq}. Reciprocally, if $r$ satisfies equation \eqref{eq:DiffEq} with $\lambda\neq 1$, then $h=\frac{r}{1-\lambda}$ also satisfies \eqref{eq:DiffEq}.
Thus, for $\lambda\neq 1$, we can work either with the support function or the curvature radius. For $\lambda=1$, if we find a solution of equation \eqref{eq:DiffEq} for $h$, \eqref{eq:CurvatureSupport}
would give us $r=0$. Thus it is more interesting in this case to work with equation \eqref{eq:DiffEq} for the curvature radius $r$.

\begin{Proposition}
Assume that $r:\R\to\R$ is a $2\pi$-periodic solution of equation \eqref{eq:DiffEq} with $\lambda\neq 1$. Then any curve $\gamma(\theta)$ defined by equation \eqref{eq:CurvatureRadius} is closed.
\end{Proposition}
\begin{proof}
Since $h=\frac{r}{1-\lambda}$, we have that $h:\R\to\R$ is $2\pi$-periodic. From equation \eqref{eq:Gamma}, we conclude that $\gamma$ is closed.
\end{proof}

\subsection{The Euclidean case}

\paragraph{Cycloids.} Consider 
\begin{equation*}
\gamma(\theta)=R\left(2(\theta-\theta_0)-\sin(2(\theta-\theta_0)),1-\cos(2(\theta-\theta_0)) \right).
\end{equation*}
The Euclidean curvature radius is $r(\theta)=4R\sin\left(\theta-\theta_0\right)$.

\paragraph{Hypocycloids.} Let $\lambda=1-\frac{R_2}{R_1}<0$ and consider 
\begin{equation*}
\gamma(t)=R_1\left( \cos(\lambda t),\sin(\lambda t)\right)+(R_2-R_1)\left( \cos(t),\sin(t)\right)
\end{equation*}
The Euclidean curvature radius is 
$r(\theta)=C\sin\left(\frac{1-\lambda}{1+\lambda}\theta\right)$,
for some constant $C$, where $\frac{1-\lambda}{1+\lambda}=\frac{R_2}{R_2-2R_1}>1$ and $\frac{\lambda+1}{2}t=\theta$. 
If $\frac{R_2}{R_2-2R_1}=\frac{m}{n}$, for some $m,n\in\N$ without common factor, 
then $\gamma$ closes after $n$ turns and has exactly $2m$ cusps. If $\frac{R_2}{R_2-2R_1}$ is not rational
the hypocycloid does not close.

\paragraph{Epicycloids.} Let $\alpha=1+\frac{R_2}{R_1}>0$ and consider 
\begin{equation*}
\gamma(t)=-R_1\left( \cos(\alpha t),\sin(\alpha t)\right)+(R_2+R_1)\left(\cos(t),\sin(t)\right).
\end{equation*}
The Euclidean curvature radius is $r(\theta)=C\sin\left(\frac{\alpha-1}{\alpha+1}\theta\right)$,
for some constant $C$, where $\frac{\alpha-1}{\alpha+1}=\frac{R_2}{R_2+2R_1}<1$, $\frac{\alpha+1}{2}t=\theta$.
If $\frac{R_2}{R_2+2R_1}=\frac{m}{n}$, $m,n\in\N$ without common factor, then $\gamma$ closes after $n$ turns and has exactly $2m$ cusps. If $\frac{R_2}{R_2+2R_1}$ is not rational
the epicycloid does not close.

\section{ Basic properties of equation \eqref{eq:DiffEq}}

\subsection{ The associated system in the plane $\left(h,\frac{h'}{[q,q']} \right)$ }

For each $\lambda>0$, denote $h(t)=h(\lambda,h(0),h'(0))(t)$ the solution of equation \eqref{eq:DiffEq} with initial conditions $(h(0),h'(0))$ and let
$A(\lambda,t)$ be the linear map $(h(0),\frac{h'}{[q,q']}(0))\to (h(t),\frac{h'}{[q,q']}(t))$.

\begin{lemma}\label{lem:Det1}
For each $t$ and $\lambda$, $\det(A(\lambda,t))=1$.
\end{lemma}
\begin{proof}
We have that $A(\lambda,0)=Id$ and for two linearly independent solutions $(h_1,\frac{h_1'}{[q,q']})$ and $(h_2,\frac{h_2'}{[q,q']})$, 
$$
\frac{d}{dt}\left( h_1\left(\frac{h_2'}{[q,q']}\right)-h_2\left(\frac{h_1'}{[q,q']}\right) \right)=0,
$$
thus proving the lemma.
\end{proof}

In the plane $(h,\tfrac{h'}{[q,q']})$, denote by $\beta$ the angle of the vector $(h,\tfrac{h'}{[q,q']})$ with the $h$-axis. The {\it total variation} $\Delta$ of $\beta$ is given by
\begin{equation}
\Delta(\lambda,h(0),h'(0))=-\frac{1}{2\pi} \int_0^{2\pi} \frac{d\beta}{dt} dt.
\end{equation} 
When $\Delta$ is an integer, it represents the number of turns the curve $\left( h(t),\frac{h'}{[q,q']}(t) \right)$, $t\in[0,2\pi]$, makes around $(0,0)$.

Note that if $\frac{h'}{[q,q']}>0 (<0)$ the curve is going to the right (left) and in the $h$ axis the tangent is vertical. For a solution of equation \eqref{eq:DiffEq} with $\lambda>0$, if $h>0 (<0)$,
the curve is going down (up) and in the $\frac{h'}{[q,q']}$ axis the tangent is horizontal. We conclude that for closed curves, $\Delta$ is an integer and the number of zeros $k$ of $h$ and $h'$ is $2\Delta$.

\begin{lemma}\label{lemma:ICrescent}
For a solution of equation \eqref{eq:DiffEq}, $\frac{d\beta}{d\theta}>0$. Moreover, the total variation $\Delta(\lambda,h(0),h'(0))$ is strictly increasing in $\lambda$.
\end{lemma}
\begin{proof}
Observe that we can calculate $\frac{d\beta}{d\theta}$ as follows:
$$
\frac{d\beta}{d\theta}=-\frac{ h\left(\frac{h'}{[q,q']}\right)'  -\frac{h'}{[q,q']} h'}{h^2+(\frac{h'}{[q,q']})^2}=\frac{\lambda h^2[p,p']+\frac{(h')^2}{[q,q']}}{h^2+(\frac{h'}{[q,q']})^2}=
$$
$$
\lambda [p,p']\cos^2(\beta)+[q,q']\sin^2(\beta).
$$
Thus $\frac{d\beta}{d\theta}>0$ and $\Delta$ is strictly increasing in $\lambda$.
\end{proof}

\begin{corollary}\label{cor:Comparison}
Take $(\lambda,h)$ solution of equation \eqref{eq:DiffEq} with two consecutive zeros $\theta_0$ and $\theta_1$ of $h'$ and let $(\bar\lambda,\bar{h})$ be another solution with a zero of $\bar{h}'$ 
at $\theta_0$. 
If $\lambda<\bar\lambda$, then $\bar{h}'$ has a zero in $(\theta_0,\theta_1)$. 
\end{corollary}
\begin{proof}
We have that $\Delta(\lambda,h,\theta_0,\theta_1)=1/2$. Thus $\Delta(\bar\lambda, \bar{h},\theta_0,\theta_1)>1/2$, thus proving the corollary.
\end{proof}

\subsection{Non-positive eigenvalues}

\begin{lemma}\label{lemma:Eigenvalue0}
The eigenvector $h(\theta)=1$ is associated to $\lambda=0$. There are no other closed eigenvectors. 
\end{lemma}
\begin{proof}
The general solution of equation \eqref{eq:DiffEq} with $\lambda=0$ is
$$
h(t)=C_1\int_0^{t} [q,q']dt+C_2. 
$$
This solution is not closed unless $C_1=0$. 
\end{proof}

\begin{lemma}
For $\lambda<0$, there are no closed eigenvectors. 
\end{lemma}
\begin{proof}
A closed non-constant curve in $(h,h')$ plane must have a vertical tangent. At such a point $h'=0$. If there were another zero of $h'$, Corollary \ref{cor:Comparison} would imply a solution $(0,\bar{h})$ 
with two zeros of $\bar{h}'$, a contradiction with Lemma \ref{lemma:Eigenvalue0}.
\end{proof}

\subsection{Cusps and orientation}

\begin{Proposition}
The cycloids associated to any solution of equation \eqref{eq:DiffEq} are regular except at isolated ordinary cusps.
\end{Proposition}

\begin{proof}
Since
$$
\gamma'(\theta)= r(\theta) p'(\theta),
$$
$\gamma$ is regular at $\theta$ if and only if $r(\theta)\neq 0$. Differentiating the above equation and making $r(\theta)=0$ we get
$$
\gamma''(\theta)= r'(\theta)p'(\theta).
$$
If $r=r'=0$, then the solution of equation \eqref{eq:DiffEq} is $r=0$. Excluding this trivial case, $r(\theta)=0$ implies $\gamma''(\theta)\neq 0$. From Example \ref{ex:Cusps}
we conclude that $\gamma$ at this point has an ordinary cusp.
\end{proof}

Next lemma says that the orientation of hypocycloids are negative, while the orientation of epicycloids are positive. 

\begin{lemma}
Let $\gamma$ be an eigenvector of $T$ with eigenvalue $\lambda>0$. Then $[\gamma,\gamma']$ has constant sign, positive if $0<\lambda<1$ and negative if $\lambda>1$.
\end{lemma}

\begin{proof}
By equations \eqref{eq:Gamma} and \eqref{eq:GammaLinha}, we have
$$
[\gamma,\gamma']=h\left(h[p,p']+\left( \frac{h'}{[q,q']} \right)'\right).
$$
Since $\gamma$ is an eigenvector of $T$, 
$$
\left( \frac{h'}{[q,q']} \right)'=-\lambda[p,p']h,
$$
and so
$$
[\gamma,\gamma']=h^2[p,p'](1-\lambda),
$$
thus proving the lemma.
\end{proof}

\subsection{ Minkowskian Cycloids}\label{sec:MinkCyc}

In this section we solve equation \eqref{eq:DiffEq} with $\lambda=1$. In this case it is more interesting to consider the curvature radius $r$ than the support function $h$ (see section
 \ref{sec:DoubleEvolutes}). 
 
\begin{Proposition}\label{Prop:Eigenvalue1}
Any function of the form $r(\theta)=[v, q(\theta)]$, $v\in\R^2$, satisfies equation \eqref{eq:DiffEq} with $\lambda=1$. Thus, given initial conditions $r(\theta_0)$ and $r'(\theta_0)$,
the solution is 
\begin{equation}\label{eq:SolEig1}
r(\theta)=\frac{1}{[q,q'](\theta_0)}\left(   [q(\theta_0),q(\theta)]r'(\theta_0)-[q'(\theta_0),q(\theta)]r(\theta_0)    \right).
\end{equation}
As a consequence, 
$$
r(\theta_0+\pi)=-r(\theta_0),\ \ \ r'(\theta_0+\pi)=-r'(\theta_0),
$$
and so $\lambda=1$ is a double eigenvalue of equation \eqref{eq:DiffEq}. The corresponding cycloids are not closed. 
\end{Proposition}

\begin{proof}
Differentiating equation \eqref{eq:SolEig1} and dividing by $[q,q'](\theta)$ we obtain
\begin{equation*}
\frac{r'(\theta)}{[q,q'](\theta)}=\frac{1}{[q,q'](\theta_0)}\left(  - [q(\theta_0),p(\theta)]r'(\theta_0)+[q'(\theta_0),p(\theta)]r(\theta_0)    \right).
\end{equation*}
Differentiating this equation and dividing by $[p,p'](\theta)$ we get
\begin{equation*}
\frac{1}{[p,p'](\theta)}  \left(  \frac{r'(\theta)}{[q,q'](\theta)} \right)'=\frac{1}{[q,q'](\theta_0)}\left(  - [q(\theta_0),q(\theta)]r'(\theta_0)+[q'(\theta_0),q(\theta)]r(\theta_0)    \right).
\end{equation*}
We conclude that $r(\theta)$ defined by equation \eqref{eq:SolEig1} satisfies equation \eqref{eq:DiffEq}.

To prove the last assertion write
$$
\gamma(\theta)=\gamma(\theta_0)+\int_{\theta_0}^{\theta} r(t)p'(t)dt.
$$
Since 
$$
\int_{\theta_0}^{\theta_0+\pi} r(t)p'(t)dt=\int_{\theta_0+\pi}^{\theta_0+2\pi} r(t)p'(t)dt,
$$
we have that $\gamma(\theta_0+2\pi)=\gamma(\theta_0)$ if and only if 
\begin{equation}\label{eq:PiPer}
\int_{\theta_0}^{\theta_0+\pi} r(t)p'(t)dt=0.
\end{equation}
Choose $\theta_0$ 
such that $r(\theta_0)=0$. Then equation \eqref{eq:SolEig1} says that $r$ does not change sign in the interval $[\theta_0,\theta_0+\pi]$, which 
is a contradiction with Equation \eqref{eq:PiPer}.
\end{proof}

\begin{example}\label{ex:OpenCycLp}
Consider the the $\mathcal{L}_p$ space, whose unit circle is $|x|^p+|y|^p=1$. This unit circle can be parameterized by $(\cos(t)^{2/p},\sin(t)^{2/p})$, $0\leq t\leq 2\pi$. Then
$[p,p']=\frac{2}{p} \cos(t)^{2/p-1}\sin(t)^{2/p-1}$. Thus
$$
q=\frac{p'}{[p,p']}=\left( -\sin(t)^{2/q}, \cos(t)^{2/q} \right),
$$
where $\frac{1}{q}+\frac{1}{p}=1$. Then $[q,q']=\frac{2}{q} \cos(t)^{2/q-1}\sin(t)^{2/q-1}$. Observe that $[p,p']\cdot[q,q']=\frac{4}{pq}$.

Consider Minkowskian cycloid with $r(\frac{\pi}{4})=0$, $r'(\frac{\pi}{4})=1$. Since $[q,q'](\frac{\pi}{4})=\frac{2}{q}$, we have
$$
r(t)=\frac{q}{2}[q(\frac{\pi}{4}), q(t)].
$$
Since $\gamma'(t)=r(t)p'(t)$, we can write
$$
\gamma'(t)=c\left( \sin(t)\cos(t)+\sin(t)^{2/q+1}\cos(t)^{2/p-1}, \sin(t)\cos(t)+\sin(t)^{2/p-1}\cos(t)^{2/q+1}\right).
$$
where $c$ is a constant. Integrating we obtain $\gamma$ (see Figure \ref{fig:OpenCyc}).
\end{example}

\begin{figure}[htb]
 \centering
 \includegraphics[width=0.50\linewidth]{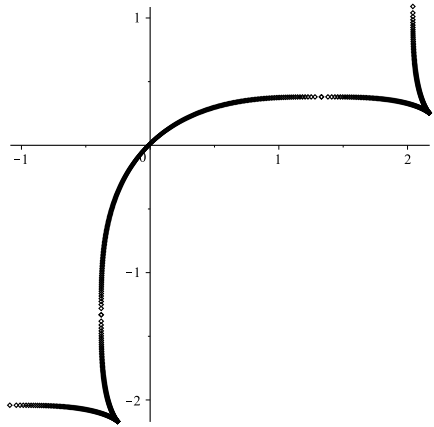}
 \caption{ The cycloid of example \ref{ex:OpenCycLp} with $p=3$. }
\label{fig:OpenCyc}
\end{figure}

\subsection{General Sturm-Liouville equation}

A general Sturm-Liouville equation is given by 
\begin{equation}\label{eq:SturmGeneral}
(au')'+(\lambda r-b)u=0,
\end{equation}
where $a>0$ and $r>0$ (see \cite[ch. 8]{CodLev}). Thus equation \eqref{eq:DiffEq} is a Sturm-Liouville equation with $b=0$, $a=\frac{1}{[q,q']}$ and $r=[p,p']$. 
It is clear that the condition $b=0$ is equivalent to $u=1$ being an eigenvector associated to the eigenvalue $\lambda=0$.
Next proposition says that the conditions $a=\frac{1}{[q,q']}$ and $r=[p,p']$ are equivalent to $\lambda=1$ being a double eigenvector of equation \eqref{eq:SturmGeneral}.

\begin{Proposition}
Consider equation \eqref{eq:SturmGeneral} with $b=0$. Then there exists a symmetric locally convex curve $p$ such that $r=[p,p']$ and $a=[q,q']^{-1}$ if and only if 
 $\lambda=1$ is a double eigenvalue of equation \eqref{eq:SturmGeneral}. 
\end{Proposition}
\begin{proof}
It follows from Proposition \ref{Prop:Eigenvalue1} that for equations of the form \eqref{eq:DiffEq}, $\lambda=1$ is a double eigenvalue. For the converse, assume that $x(t)$ and $y(t)$ 
are linearly independent solutions of equation \eqref{eq:SturmGeneral} with $\lambda=1$. Writing $q(t)=(x(t),y(t))$ we have 
$$
\frac{a'}{r}q'+\frac{a}{r}q''=-q.
$$
which implies 
\begin{equation}\label{eq:ar}
\frac{a}{r}=\frac{[q,q']}{[q',q'']};\ \ \frac{a'}{r}=-\frac{[q,q'']}{[q',q'']}.
\end{equation}
Defining $p$ by equation \eqref{eq:ParameterP}, we obtain
$[p,p'] \cdot [q,q']^2=[q',q'']$. Taking this into account, the substitution of formulas \eqref{eq:ar} in equation \eqref{eq:SturmGeneral} lead to equation \eqref{eq:DiffEq}. This means that, possibly rescaling 
and interchanging the coordinates of $q$, we may assume $a=[q,q']^{-1}$ and $r=[p,p']$ .
Finally the hypotheses $a>0$ and $r>0$ imply that $[q',q'']>0$ and $[p',p'']>0$, so $p$ and $q$ are locally convex. 
\end{proof}

\section{Sturm-Liouville theory: Closed hypocycloids and epicycloids}

\subsection{The space of $2\times 2$ matrices with determinant $1$}

Denote by ${\mathcal M}=SL_2(\R)$ the set of $2\times 2$ matrices with determinant $1$. This set can be decomposed in three connected subsets: ${\mathcal M}^{+}$, the set of matrices 
with real positive eigenvalues, ${\mathcal M}^{-}$, the set of matrices with real negative eigenvalues, and ${\mathcal M}^{0}$, the set of matrices with complex eigenvalues. 
The common boundary of ${\mathcal M}^{+}$ and ${\mathcal M}^{0}$ is the set ${\mathcal M}^{1}$ of matrices with both eigenvalues equal to $1$. The common boundary
of ${\mathcal M}^{-}$ and ${\mathcal M}^{0}$ is the set ${\mathcal M}^{-1}$ of matrices with both eigenvalues equal to $-1$.  There is no path from $\mathcal{M}^{+}$ to $\mathcal{M}^{-}$
without passing through $\mathcal{M}^{0}$.

\subsection{Main results}

We shall denote $A(\lambda)=A(\lambda, \pi)$. By lemma \ref{lem:Det1},  $A(\lambda)\in{\mathcal M}$. The eigenvalues of equation \eqref{eq:DiffEq} are the values of $\lambda$ for which $A(\lambda)\in {\mathcal M}^{1}\cup{\mathcal M}^{-1}$.

As we have seen, $0$ is a single and $1$ is a double eigenvalue of equation \eqref{eq:DiffEq}. Thus we write $\lambda_0=0$, $\lambda_1^1=\lambda_1^2=1$.
In this section we shall prove the following theorem:

\begin{thm}\label{thm:SturmLiouville}
Consider $k\geq 2$:
\begin{itemize}
\item For each $k\in N$, $k$ even, we can find eigenvalues $\lambda_k^1\leq \lambda_k^2$ such that $A(\lambda_k^i)\in {\mathcal M}^{1}$ and $A(\lambda)\in{\mathcal M}^{+}$, for each 
$\lambda\in(\lambda_k^1,\lambda_k^2)$. 
\item For each $k\in N$, $k$ odd, we can find eigenvalues $\lambda_{k}^1\leq \lambda_{k}^2$ such that $A(\lambda_{k}^i)\in {\mathcal M}^{-1}$ and $A(\lambda)\in{\mathcal M}^{-}$, for each 
$\lambda\in(\lambda_{k}^1,\lambda_{k}^2)$ and  $\lambda_1^1=\lambda_1^2=1$. 
\item For each $k\in \N$, we have $\lambda_{k-1}^2<\lambda_k^1$ and $\lambda_k^2<\lambda_{k+1}^1$. Moreover, for $\lambda\in (\lambda_{k-1}^2,\lambda_k^1)\cup (\lambda_k^2,\lambda_{k+1}^1)$, $A(\lambda)\in{\mathcal M}^{0}$.
\end{itemize}
\end{thm}

This theorem holds for general Sturm-Liouville equations \eqref{eq:SturmGeneral} (see \cite{CodLev}). We give here a geometric proof that holds only for 
Sturm-Liouville equations of type \eqref{eq:DiffEq}.

\begin{lemma}\label{lemma:Equivalences}
For each $\lambda$, consider the interval $I(\lambda)=[I_{min}(\lambda),I_{max}(\lambda)]$ of indexes when we vary $(h(0),h'(0))$. 
We have that $k/2\in I(\lambda)$ if and only if $A(\lambda)$ has an eigenvector $(h(0),h'(0))$ of index $k/2$. In this case,
$A(\lambda)\in\mathcal{M}^1$, for $k$ even, or $A(\lambda)\in\mathcal{M}^{-1}$, for $k$ odd.
\end{lemma}
\begin{proof}
First observe that $k/2\in I(\lambda)$ if and only if, for some $(h(0),h'(0))$ the index is $k/2$, which is equivalent to $(h(0),h'(0))$ being an eigenvector of $A(\lambda)$ of index $k/2$. 
\end{proof}

\begin{proof} (of Theorem \ref{thm:SturmLiouville}).
From Lemma \ref{lemma:ICrescent}, $I_{min}(\lambda)$ and $I_{max}(\lambda)$ are strictly increasing functions of $\lambda$. Define $\lambda_k^1$ and $\lambda_k^2$ as the minimum and maximum 
values of $\lambda$ such that $k\in I(\lambda)$. From Lemma \ref{lemma:Equivalences}, $A(\lambda)\in\mathcal{M}^1$, for  $\lambda\in [\lambda_k^1,\lambda_k^2]$, $k$ even,  and  $A(\lambda)\in\mathcal{M}^{-1}$, for  $\lambda\in [\lambda_k^1,\lambda_k^2]$, $k$ odd. Since any path from $\mathcal{M}^{1}$ to $\mathcal{M}^{-1}$ must necessarily pass through $\mathcal{M}^{0}$, we conclude 
that $\lambda_{k-1}^2<\lambda_k^1$ and $\lambda_k^2<\lambda_{k+1}^1$, for any $k\in\N$. 

Denote by $h_k^i$ the $\lambda_k^i$-eigenvectors. For $k$ odd, since equation \eqref{eq:DiffEq} is linear and $\pi$-periodic, the extension $h_k^i(\theta+\pi)=-h_k^i(\theta)$ to 
$[0,2\pi]$ is periodic of index $k$. For $k$ even, the extension $h_k^i(\theta+\pi)=h_k^i(\theta)$ is $\pi$-periodic and hence $2\pi$-periodic. 
The vectors $\left(h_k^i(0), (h_k^i)'(0)\right)$ are the unique eigenvectors of $A(\lambda,2\pi)$ with eigenvalue $1$ and index $k$. 
In fact, the same considerations as above show that for any $k$, there are exactly two values of $\lambda$ such that the linear transformation 
$A(\lambda, 2\pi)$ admits the eigenvalue $1$ with index $k$, and $\lambda=\lambda_k^i$ do this job. 
\end{proof}

\begin{figure}[htb]
 \centering
 \includegraphics[width=0.50\linewidth]{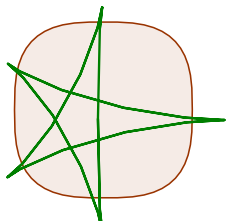}
 \caption{ A closed hypocycloid for the $\mathcal{L}_p$ space ($p=3$) with $\lambda_k^1=\lambda_k^2=27.1$ and $k=5$ (see proposition \ref{prop:DoubleEig}).  }
\label{fig:ClosedCyc}
\end{figure}

\begin{corollary}
For $\lambda\in(\lambda_k^1,\lambda_k^2)$, $k\in\N$, the $\mathcal{P}$-cycloid is unbounded. 
\end{corollary}
\begin{proof}
For $\lambda\in(\lambda_k^1,\lambda_k^2)$, $A(\lambda,\pi)$ has a real eigenvalue of absolute value bigger that $1$. Since $A(\lambda,N\pi)=A(\lambda,\pi)^N$, the corollary is proved.
\end{proof}

\begin{figure}[htb]
 \centering
 \includegraphics[width=0.50\linewidth]{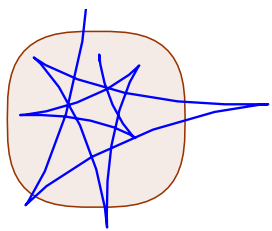}
 \caption{ An unbounded hypocycloid for the $\mathcal{L}_p$ space ($p=3$) with $\lambda=19.79$. }
\label{fig:UnboundedCyc}
\end{figure}

\subsection{Closing after $N$ turns}

We now consider curves in $\mathcal{H}$ that close after $N$ turns. Such curves are called $N$-hedgehogs (\cite{Martinez-Maure}).
We look for $N$-hedgehogs that are hypocycloids and epicycloids, i.e., that satisfy equation \eqref{eq:DiffEq} for some $\lambda>1$ or $\lambda<1$,
respectively.   

\begin{thm} Fix $N\in\N$, $N>1$. 
\begin{enumerate}
\item There exist $0<\lambda_{1,N}<...<\lambda_{{N-1},N}<1$ double eigenvalues of the $2\pi N$-periodic problem associated to equation \eqref{eq:DiffEq}. 
Each eigenvalue $\lambda_{k,N}$ determines a pair of linearly independent epicycloids with $2k$ cusps.  
\item For $k>N$, $(k,N)$ coprime, there exists a double eigenvalue $\lambda_{k,N}$  of the $2\pi N$-periodic problem associated to equation \eqref{eq:DiffEq}. 
Each eigenvalue $\lambda_{k,N}$ determines a pair of linearly independent hypocycloids with $2k$ cusps.  
\item For $k>N$, $\frac{k}{N}=\frac{k_1}{N_1}$, $(k_1,N_1)$ coprime, there exists a pair of linearly independent hypocycloids with $2k$ cusps 
obtained by traversing $\frac{N}{N_1}$ times the $2\pi N_1$-periodic  hypocycloids with $2k_1$ cusps.  
\end{enumerate}
\end{thm}

\begin{proof}
To prove item 1, consider $\lambda$ in the interval $(0,1)$. Since $A(\lambda)\in{\mathcal M}^0$, there exist exactly $N-1$ values of $\lambda$
such that $A(\lambda)^N=Id$. We denote them $\lambda_{k,N}$, $k=1,..,N-1$. Each $\lambda_{k,N}$ is a double eigenvalue of the $2\pi N$-periodic problem associated to equation \eqref{eq:DiffEq} and determines a pair of linearly independent closed epicycloids with $2k$ cusps.  
The proof of item 2 is similar to the proof of item 1. Item 3 is obvious.
\end{proof}

\subsection{Symmetry and double eigenvalues}

\begin{Proposition}\label{prop:DoubleEig}
Assume that the unit ball of $\mathcal{P}$ is invariant under some linear transformation $B:\R^2\to\R^2$ such that $B^2=-I$. Then 
$\lambda_k^1=\lambda_k^2$, for any $k$ odd. 
\end{Proposition}

\begin{proof}
By choosing an adequate basis of $\R^2$, we may assume that $B$ is a $\pi/2$ rotation, which implies that $[p,p']$ and $[q,q']$ are $\pi/2$-periodic, and so does equation \eqref{eq:DiffEq}.
We conclude that $A(\lambda,\pi)=A(\lambda,\pi/2)^2$. 

As in the proof of Theorem \ref{thm:SturmLiouville}, choose $\mu_k^i$ such that $A(\mu_k^i,\pi/2)\in\mathcal{M}^{-1}$, $k$ odd, and $A(\mu_k^i,\pi/2)\in\mathcal{M}^{1}$, $k$ even.
Take $\lambda_k\in(\mu_{k-1}^2,\mu_k^1)$ such that $A(\lambda,\pi/2)^2=-Id$.  We conclude that 
$A(\lambda_k,\pi)=-Id$. Thus $\lambda_k^1=\lambda_k^2=\lambda_k$. 
\end{proof}

The above result says that double eigenvalues appear in the presence of symmetry. The following conjecture is a kind of converse:

\medskip\noindent
{\bf Conjecture:} Assume that  $\lambda_k^1=\lambda_k^2=k^2$, for all $k\in\N$. Then the unit ball is Euclidean. 

\begin{remark}
It is a famous conjecture whether or not the eigenvalues of the Laplacian of a planar region determine the region itself, up to rigid motions (\cite{Kac}).
The above conjecture for convex curves has a similar flavour.
\end{remark}

\section{ An orthonormal basis for ${\mathcal C}^0(S^1)$}

\subsection{The kernel of the operator $T$}

Consider ${\mathcal C}^0(S^1)$ the space of $h:S^1\to\R$ with the inner product
$$
\left< h_1,h_2\right>=\int_0^{2\pi} h_1(\theta)h_2(\theta)[p,p']d\theta.
$$
\begin{lemma}
The linear mapping 
$$
Th=-\frac{1}{[p,p']}\left(  \frac{h'}{[q,q']} \right)'
$$
is self-adjoint with respect to this inner product.
\end{lemma}
\begin{proof}
Observe that
$$
\left< h_1,Th_2\right>=\int_0^{2\pi} h_1Th_2[p,p']d\theta=-\int_0^{2\pi} h_1\left(\frac{h_2'}{[q,q']}\right)'d\theta=\int_0^{2\pi} \frac{h_1'h_2'}{[q,q']}d\theta
$$
which is equal to $\left< Th_1,h_2\right>$.
\end{proof}

The kernel $K$ of $T$ is the $1$-dimensional subspace of constant functions. Denote $L_0=K^{\perp}$ and observe that, for $h\in L_0$,
\begin{equation}\label{eq:ZeroDualLength}
\int_0^{2\pi} h(\theta)[p,p'](\theta)d\theta=0.
\end{equation}
Denote by $S:L_0\to {\mathcal C}^0(S^1)$ the inverse of $T$.

\subsection{Compactness of the inverse operator}

\begin{lemma}
For $h\in L_0$, let $g=Sh$. Then 
$$
||g||_{\infty}\leq ||[q,q']||_2 ||[p,p']||_2 ||h||_2;\ \  ||g'||_{\infty}\leq ||[q,q']||_{\infty} ||[p,p']||_2 ||h||_2.
$$
\end{lemma}
\begin{proof}
Write
$$
f(t)=\int_0^t h(t)[p,p'](t)dt,\ \ g(t)=\int_0^t f(t)[q,q'](t)dt,
$$
where $f,g \in L_0$. Then
$$
||g'||_{\infty}=||f[q,q']||_{\infty}\leq ||f||_{\infty} ||[q,q']||_{\infty}\leq ||[q,q']||_{\infty} ||[p,p']||_{2} ||h||_2
$$
and
$$
||g||_{\infty}\leq ||[q,q']||_2 ||f||_2\leq ||[q,q']||_2 ||f||_{\infty}\leq ||[q,q']||_2 ||[p,p']||_2 ||h||_2.
$$
\end{proof}

\begin{Proposition}
$S:(L_0,||\cdot||_2)\to (L_0,||\cdot||_2)$ is a compact operator. 
\end{Proposition}
\begin{proof}
Let $\{h_n\}$ be a bounded sequence in $(L_0,||\cdot||_2)$. By the above lemma, $g_n=Sh_n$ is equicontinuous and uniformly bounded.
Thus we can find an uniformly convergent subsequence $\{g_{n_j}\}$. This subsequence is also convergent in the $||\cdot||_2$ norm.
\end{proof}

\begin{corollary}
The set $\{h_k^i\}$, $k\in\N$, $i=1,2$, form an orthonormal basis of ${\mathcal C}^0(S^1)$. 
\end{corollary} 

\subsection{Support functions of closed curves in $\H$}

Assume now that $h$ is the support function of a closed curve $\gamma\in\H$. Then $K$ represents the multiples of the unit ball $\mathcal{P}$, 
while $L_0$ defined by equation \eqref{eq:ZeroDualLength} represents curves $\gamma\in\H$ whose dual length is zero.
By a translation in the plane, we may write 
$$
h=h_0+\sum_{k\geq 2} a_{k}^i h_{k}^i.
$$
where $h_0\in K$. Denote by $C_0$ the subspace of $L_0$ generated by ${h_k^i}$, $k\geq 2$. 

Denote by $E_0$ the subspace of $C_0$ generated by ${h_k^i}$, $k$ even, and by $W_0$ the subspace of $C_0$ generated by ${h_k^i}$, $k>1$ odd.  
Then $E=E_0+K$ is the subspace of support functions of symmetric closed curves, while $W=W_0+K$ corresponds to support functions of constant width 
closed curves in $\H$.

\subsection{Convergence of involutes iteration}

In \cite{Craizer14}, it is proved that the iterations of involutes of a constant width curve converge to a constant curve. This result 
can be re-phrased by saying that the iterations $S^nh$ converge to $0$, for any $h\in W_0$. %a linear combination of $h_1^1$ and $h_1^2$. 
We can recover this result by using the basis $\{h_k^i\}$. In fact, we shall prove that for any $h\in C_0$, not necessarily in $W_0$, $S^nh$ converges to $0$, 
which means that the iteration of involutes of any closed curve $\gamma\in\H$ with dual length $0$ converges to $0$ (see \cite{AFITT} for the euclidean 
case).

Let $h$ be the support function of a closed curve $\gamma\in C_0$. We can write 
$$
h=\sum_{k\geq 2} a_k^i h_k^i.
$$
Then 
$$
S^n h= \sum_{k\geq 2} \frac{a_k^i}{(\lambda_k^i)^n} h_k^i.
$$
Since $\lambda _k^i>1$, for $k\geq 2$, we conclude that $S^nh$ converges to $0$. We remark that for $n$ 
large enough, the shape of the $n$-th involute of $\gamma$ is close to the shape of the cycloid corresponding to $h_k^i$ associated
with the smallest non-zero $a_k^i$ of $h$.

\subsection{Generalized Sturm-Hurwitz Theorem}

The classical Sturm-Hurwitz theorem says that a function whose first $N$ harmonics are zero must have at least $2N$ zeros (\cite{Martinez-Maure}).

\begin{thm}\label{thm:Sturm-Hurwitz}
Write
$$
h=\sum_{k\geq k_0} a_k^ih_k^i,
$$
for some $k_0\geq 2$. Then $h$ has at least $2k_0$ zeros.
\end{thm}
\begin{proof}
Denote by $\#(h)$ the number of zeros of $h$. By Rolle's theorem, $\#(h')\geq \#(h)$. The same reasoning implies that $\#(Th)\geq \#(h)$. We conclude 
that $\#(Sh)\leq \#(h)$. Assume $a_{k_0}^1\neq 0$.
$$
(\lambda_{k_0}^1)^{n} S^nh =   a_{k_0}^1h_{k_0}^1+\frac{(\lambda_{k_0}^1)^{n}}{(\lambda_{k_0}^2)^{n}}a_{k_0}^2h_{k_0}^2 +\sum_{k>k_0} \frac{(\lambda_{k_0}^1)^{n}}{(\lambda_{k}^i)^{n}}a_{k}^ih_{k}^i  \ .
$$
Taking $n$ sufficiently large, the second member must have the same number of zeros as $h_{k_0}^1$, i.e., $2k_0$. Thus the number of zeros of $h$ is at least $2k_0$. 
\end{proof}

\subsection{A four and a six vertices theorem}

In this section we need to consider the dual subspaces, i.e., subspaces associated with the dual unit ball $q$, and we shall denote them by a $*$.

A vertex of $\gamma\in\H$ is a point where the derivative of the curvature radius vanishes.

\begin{thm}
Let $\gamma\in\H$ be a closed curve. Then $\gamma$ has at least four vertices.
\end{thm}
\begin{proof}
By equation \eqref{eq:CurvEvolute}, zeros of the curvature radius $r_{\delta}$ of the evolute $\delta$ correspond to vertices of $\gamma$. 
But since $r_{\delta}\in C_0^*$, theorem \ref{thm:Sturm-Hurwitz} says that $r_{\delta}$ has at least $4$ zeros. 
\end{proof}

\begin{thm}
Let $\gamma\in\H$ be a closed curve with constant width. Then $\gamma$ has at least six vertices.
\end{thm}

\begin{proof}
If $\gamma\in W_0$,  then $\delta\in W_0^*$ (see \cite{Craizer14}) and so $h_{\delta}\in W_0^*$. Now equation \eqref{eq:CurvatureSupport} says that
$r_{\delta}=T^{*}h_{\delta}$, and since $W_0^*$ is $T^{*}$-invariant, we conclude that $r_{\delta}\in W_0^*$.
By theorem \ref{thm:Sturm-Hurwitz}, $r_{\delta}$ has at least $6$ zeros.  
\end{proof}

% ------------------------------------------------------------------------
\end{document}